\documentclass[11pt]{article}
\usepackage{amsfonts,amsmath}
\usepackage{hyperref,amsthm}
\usepackage{epsfig, amssymb,amsmath}

\newtheorem{theorem}{Theorem}[section]
\newtheorem{definition}[theorem]{Definition}
\newtheorem{proposition}[theorem]{Proposition}
\newtheorem{lemma}[theorem]{Lemma}
\newtheorem{claim}[theorem]{Claim}

\newtheorem{observation}[theorem]{Observation}
\newtheorem{conjecture}[theorem]{Conjecture}

\renewenvironment{proof}[1][Proof]{ \noindent \textbf{#1: }}{$\Box$
\bigskip}

\newcommand{\Ent}{{\rm Ent}}
\newcommand{\Inf}{{\rm I}}

\begin{document}

\title{A Note on the Entropy/Influence Conjecture}

\author{
Nathan Keller\thanks{Weizmann Institute of Science. Partially
supported by the Koshland Center for Basic Resarch. E-mail:
nathan.keller@weizmann.ac.il.}, Elchanan Mossel\thanks{U.C.
Berkeley and Weizmann Institute of Science. Supported by DMS
0548249 (CAREER) award, by DOD ONR grant N000141110140, by ISF
grant 1300/08 and by a Minerva Grant. E-mail:
mossel@stat.berkeley.edu.}, and Tomer Schlank\thanks{Hebrew
University. Partially supported by the Hoffman program for
leadership. E-mail: tomer.schlank@gmail.com.} }

\maketitle

\begin{abstract}

The entropy/influence conjecture, raised by Friedgut and
Kalai~\cite{Friedgut-Kalai} in 1996, seeks to relate two different
measures of concentration of the Fourier coefficients of a Boolean
function. Roughly saying, it claims that if the Fourier spectrum
is ``smeared out'', then the Fourier coefficients are concentrated
on ``high'' levels. In this note we generalize the conjecture to
biased product measures on the discrete cube, and prove a variant
of the conjecture for functions with an extremely low Fourier
weight on the ``high'' levels.

\end{abstract}

\section{Introduction}

\begin{definition}
Consider the discrete cube $\{0,1\}^n$ endowed with the product
measure $\mu_p=(p \delta_{\{1\}} + (1-p) \delta_{\{0\}})^{\otimes
n}$, denoted in the sequel by $\{0,1\}^n_p$, and let
$f:\{0,1\}^n_p \rightarrow \mathbb{R}$. The Fourier-Walsh
expansion of $f$ with respect to the measure $\mu_p$ is the unique
expansion
\[
f = \sum_{S \subset \{1,2,\ldots,n\}} \alpha_S u_S,
\]
where for any $T \subset \{1,2,\ldots,n\}$,\footnote{Throughout
the paper, we identify elements of $\{0,1\}^n$ with subsets of
$\{1,2,\ldots,n\}$ in the natural way.}
\[
u_S(T)=\Big(-\sqrt{\frac{1-p}{p}} \Big)^{|S \cap T|}
\Big(\sqrt{\frac{p}{1-p}} \Big)^{|S \setminus T|}.
\]
In particular, for the uniform measure (i.e., $p=1/2$),
$u_S(T)=(-1)^{|S \cap T|}$. The coefficients $\alpha_S$ are
denoted by $\hat f(S)$,\footnote{Note that since the functions
$\{u_S\}_{S \subset \{1,\ldots,n\}}$ form an orthonormal basis,
the representation is indeed unique, and the coefficients are
given by the formula $\hat f(S) = \mathbb{E}_{\mu_p} [f \cdot
u_S]$ .} and the level of the coefficient $\hat f(S)$ is $|S|$.
\end{definition}

Properties of the Fourier-Walsh expansion are one of the main
objects of study in discrete harmonic analysis. The entropy/influence
conjecture, raised by Friedgut and Kalai~\cite{Friedgut-Kalai} in
1996, seeks to relate two measures of concentration of the Fourier
coefficients (i.e. coefficients of the Fourier-Walsh expansion) of
Boolean functions. The first of them is the \emph{spectral entropy}.
\begin{definition}
Let $f:\{0,1\}^n_p \rightarrow \{-1,1\}$ be a Boolean function.
The spectral entropy of $f$ with respect to the measure $\mu_p$ is
\[
\Ent_p(f)=\sum_{S \subset \{1,\ldots,n\}} \hat f(S)^2 \log
\left(\frac{1}{\hat f(S)^2} \right),
\]
where the Fourier-Walsh coefficients are computed w.r.t. to
$\mu_p$.
\end{definition}

Note that by Parseval's identity, for any Boolean function we have
$\sum_S \hat f(S)^2 =1$, and thus, the squares of the Fourier
coefficients can be viewed as a probability distribution on the
set $\{0,1\}^n$. In this notation, the spectral entropy is simply
the entropy of this distribution, and intuitively, it measures how
much are the Fourier coefficients ``smeared out''.

\medskip

The second notion is the \emph{total influence}.
\begin{definition}
Let $f:\{0,1\}^n_p \rightarrow \{0,1\}$. For $1 \leq i \leq n$,
the influence of the $i$-th coordinate on $f$ with respect to
$\mu_p$ is
\[
\Inf_i^p(f)=\Pr_{x \sim \mu_p} [f(x) \neq f(x \oplus e_i)],
\]
where $x\oplus e_i$ denotes the point obtained from $x$ by
replacing $x_i$ with $1-x_i$ and leaving the other coordinates
unchanged.

\medskip

\noindent The total influence of the function $f$ is
\[
\Inf_p(f)=\sum_{i=1}^n \Inf_i^p(f).
\]
\end{definition}

Influences of variables on Boolean functions were studied extensively
in the last decades, and have applications in a wide variety of fields,
including Theoretical Computer Science, Combinatorics, Mathematical
Physics, Social Choice Theory, etc. (see, e.g., the survey~\cite{Kalai-Safra}.)
As observed in~\cite{KKL}, the total influence can be expressed
in terms of the Fourier coefficients:
\begin{observation}
Let $f:\{0,1\}^n_p \rightarrow \{-1,1\}$. Then
\begin{equation}\label{Eq:Intro-new1}
\Inf_p(f)=\frac{1}{4p(1-p)} \sum_S |S| \hat f(S)^2.
\end{equation}
In particular, for the uniform measure $\mu_{1/2}$,
$\Inf_{1/2}(f)=\sum_S |S| \hat f(S)^2$.
\end{observation}

Thus, in terms of the distribution induced by the Fourier
coefficients, the total influence is (up to normalization) the
\emph{expectation of the level} of the coefficients, and it
measures the question whether the coefficients are concentrated on
``high'' levels.

\medskip

The entropy/influence conjecture asserts the following:
\begin{conjecture}[Friedgut and Kalai]
Consider the discrete cube $\{0,1\}^n$ endowed with the uniform
measure $\mu_{1/2}$. There exists a universal constant $c$, such
that for any $n$ and for any Boolean function $f:\{0,1\}^n_{1/2}
\rightarrow \{-1,1\}$,
\[
\Ent_{1/2}(f) \leq c \cdot \Inf_{1/2}(f).
\]
\label{Conj:Original}
\end{conjecture}

The conjecture, if confirmed, has numerous significant
implications. For example, it would imply that for any property of
graphs on $n$ vertices, the sum of influences is at least $c (\log
n)^2$ (which is tight for the property of containing a clique of
size $\approx \log n$). The best currently known lower bound, by
Bourgain and Kalai~\cite{Bourgain-Kalai1}, is $\Omega((\log
n)^{2-\epsilon})$, for any $\epsilon>0$.

Another consequence of the conjecture would be an affirmative
answer to a variant of a conjecture of Mansour~\cite{Mansour} stating that if a
Boolean function can be represented by a DNF formula of polynomial
size in $n$ (the number of coordinates), then most of its Fourier
weight is concentrated on a polynomial number of coefficients (see~\cite{Ryan}
for a detailed explanation of this application). This conjecture,
raised in 1995, is still wide open.

\medskip

In this note we explore the entropy/influence conjecture in two directions:

\medskip

\noindent \textbf{Biased measure on the discrete cube.} We state a
generalization of the conjecture to the product measure $\mu_p$ on
the discrete cube:
\begin{conjecture}\label{Conj:Biased}
There exists a universal constant $c$, such that for any $0<p<1$,
for any $n$ and for any Boolean function $f:\{0,1\}^n_p
\rightarrow \{-1,1\}$,
\[
\Ent_p(f) \leq c p \log(1/p) \cdot \Inf_p(f).
\]
\end{conjecture}
We prove that Conjecture~\ref{Conj:Biased} follows from the
original Entropy/Influence conjecture, and that it is tight for
the graph property of containing a clique of fixed size (at the
critical probability). This answers a question raised by
Kalai~\cite{Kalai-Blog}.

\medskip

\noindent \textbf{Functions with a low Fourier weight on the ``high''
levels.} We consider a weaker version of the conjecture stating
that if ``almost all'' the Fourier weight of a function is
concentrated on the lowest $k$ levels, then its entropy is at most
$c \cdot k$. We prove this statement in an extreme case:
\begin{proposition}\label{Prop:Intro1}
Let $f:\{0,1\}^n_{1/2} \rightarrow \{-1,1\}$ be a Boolean function
such that all the Fourier weight of $f$ is concentrated on the
first $k$ levels. Then all the Fourier coefficients of $f$ are of
the form
\[
\hat f(S) = a(S) \cdot 2^{-k}
\]
where $a(S) \in \mathbb{Z}$. In particular, $\Ent_{1/2}(f) \leq
2k$.
\end{proposition}
Then we cite a stronger unpublished result of Bourgain and
Kalai~\cite{Bourgain-Kalai2} which shows that if the Fourier
weight beyond the $k$th level decays exponentially, then the
spectral entropy is bounded from above by $c \cdot k$.

Finally, we suggest that if one could generalize the result of
Bourgain and Kalai to {\it some} slower rate of decay, this would
lead to a proof of the entire entropy/influence conjecture, using a
{\it tensorisation} technique.





This note is organized as follows. In Section~\ref{sec:biased} we
consider the generalization of the entropy/influence conjecture to
the biased measure on the discrete cube. Functions with a low
Fourier weight on the high levels are discussed in
Section~\ref{sec:high-levels}. We conclude the paper with an easy proof
of a weaker upper bound on the entropy, and with a connection between
the entropy/influence conjecture and Friedgut's characterization
of functions with a low total influence~\cite{Friedgut1} in
Section~\ref{sec:remarks}.

\section{Entropy/Influence Conjecture for the Product Measure $\mu_p$ on the
Discrete Cube} \label{sec:biased}

In this section we consider the space $\{0,1\}^n_p$, for $0<p<1$.
First we formulate a variant of the entropy/influence conjecture
for the biased measure and prove that it follows from the original
conjecture. Then we show that it is tight for the graph property
of containing a copy of a complete graph $K_r$ as an induced
subgraph, for random graphs distributed according to the model
$G(n,p)$, at the critical probability $p_c$.
\begin{proposition}\label{Prop:Biased}
Assume that the entropy/influence conjecture holds. Then there
exists a universal constant $c$ such that for any $0<p<1$, for any
$n$ and for any $f:\{0,1\}^n_p \rightarrow \{-1,1\}$, we have
\[
\Ent_p(f) \leq c p \log(1/p) \cdot \Inf_p(f).
\]
\end{proposition}

Our proof is based on a standard reduction from the biased measure
$\mu_p$ to the uniform measure $\mu_{1/2}$ first considered
in~\cite{BKKKL}. Let $p \leq 1/2$, and assume that
$p=t/2^m$.\footnote{It is clear that there is no loss of
generality in assuming that $p$ is diadic, as the results for
general $p$ follow immediately by approximation.} For any function
$f:\{0,1\}^n \rightarrow \mathbb{R}$ we define a function
$Red(f)=g:\{0,1\}^{mn} \rightarrow \mathbb{R}$ as follows: each $y
\in \{0,1\}^{mn}$ is considered as a concatenation of $n$ vectors
$y^i \in \{0,1\}^m$, and each such vector is translated to a
natural number $0 \leq Bin(y^i) < 2^m$ through its binary
expansion (i.e., $Bin(y^i)=\sum_{j=0}^{m-1} 2^j \cdot y^i_{m-j}$).
Then, for any $y \in \{0,1\}^{mn}$,
\[
g(y)=g(y^1,y^2,\ldots,y^n):=f \left(h(y^1),h(y^2),\ldots,h(y^n)
\right),
\]
where $h:\{0,1\}^m \rightarrow \{0,1\}$ is given by
\[
h(y^i) = \left\lbrace
  \begin{array}{c l}
    1, & Bin(y^i) \geq 2^m-t \\
    0, & Bin(y^i) < 2^m-t.
  \end{array}
\right.
\]
We use two simple properties of the reduction. The first, proved
by Friedgut and Kalai~\cite{Friedgut-Kalai}, relates the total
influence of $g$ (w.r.t. $\mu_{1/2}$) to that of $f$ (w.r.t. to
$\mu_p$).
\begin{lemma}[Friedgut and Kalai]
Let $f:\{0,1\}^n_p \rightarrow \{-1,1\}$, and let $g=Red(f)$. Then
\begin{equation}\label{Eq:Red-FK}
\Inf_{1/2}(g) \leq 6 p \lfloor \log(1/p) \rfloor \Inf_p(f).
\end{equation}
\end{lemma}

The second property relates the Fourier coefficients of $f$
(w.r.t. $\mu_{p}$) to corresponding coefficients of $g$ (w.r.t.
$\mu_{1/2}$).
\begin{lemma}
Let $f:\{0,1\}^n_p \rightarrow \mathbb{R}$, and let $g=Red(f)$.
For any $S \subset \{1,2,\ldots,mn\}$, denote $S_i=S \cap
\{(i-1)m+1,(i-1)m+2,\ldots,im\}$, and for $S' \subset
\{1,2,\ldots,n\}$, let
\[
V(S')=\{S \subset \{1,2,\ldots,mn\} : \{i: |S_i|>0\} = S'\}.
\]
Then:
\begin{equation}\label{Eq:Red0}
\sum_{S \in V(S')} \hat g(S)^2 = \hat f(S')^2.
\end{equation}
\end{lemma}
\begin{proof}
For each $S' \subset \{1,2,\ldots,n\}$, let $f_{S'}:\{0,1\}^n_p
\rightarrow \mathbb{R}$ be defined by $f_{S'}=\hat f(S') u_{S'}$.
We claim that
\begin{equation}\label{Eq:Red1}
Red(f_{S'}) = \sum_{S \in V(S')} \hat g(S) u_S.
\end{equation}
This claim implies the assertion, as by the Parseval identity,
Equation~(\ref{Eq:Red1}) implies:
\[
\sum_{S \in V(S')} \hat g(S)^2 = ||Red(f_{S'})||_2^2 =
||f_{S'}||_2^2 = \hat f(S')^2.
\]
(The first and third equalities use the Parseval identity, and the
middle equality holds since by the structure of the reduction, it
preserves all $L_p$ norms.)

\medskip

\noindent In order to prove Equation~(\ref{Eq:Red1}), we use
Proposition~2.2 in~\cite{Reduction} that describes the exact
relation between the Fourier coefficients of $Red(f)$ and the
corresponding coefficients of $f$. By the proposition, for all $S
\in V(S')$,
\[
\widehat{Red(f)}(S) = c(S,p) \cdot \hat f(S'),
\]
where $c(S,p)$ depends on $S$ and $p$ but not on $f$. Hence, for
all $S \in V(S')$, we have $\widehat{Red(f_{S'})}(S) =
\widehat{Red(f)}(S)$ (since both are determined by $S,p$, and
$\hat f(S')$). Similarly, for all $S \not \in V(S')$,
$\widehat{Red(f_{S'})}(S)=0$, since $\widehat{f_{S'}}(S'') = 0$
for all $S'' \neq S'$. Therefore, the Fourier expansion of
$Red(f_{S'})$ is:
\[
Red(f_{S'}) = \sum_{S \in V(S')} \widehat{Red(f)}(S) u_S,
\]
as asserted.
\end{proof}

\noindent Now we are ready to prove Proposition~\ref{Prop:Biased}.

\medskip

\begin{proof}
Let $f:\{0,1\}^n_p \rightarrow \{-1,1\}$, and let $g=Red(f)$. By
Equation~(\ref{Eq:Red0}),
\begin{align}\label{Eq:Red2}
\begin{split}
\Ent_{1/2}(g) =& \sum_{S \subset \{1,2,\ldots, mn\}} \hat g(S)^2
\log \frac{1}{\hat g(S)^2} = \sum_{S' \subset \{1,2,\ldots,n\}}
\sum_{S
\in V(S')} \hat g(S)^2 \log \frac{1}{\hat g(S)^2} \\
\geq& \sum_{S' \subset \{1,2,\ldots,n\}} \sum_{S \in V(S')} \hat
g(S)^2 \log \frac{1}{\hat f(S')^2} = \sum_{S' \subset
\{1,2,\ldots,n\}} \hat f(S')^2 \log \frac{1}{\hat f(S')^2} =
\Ent_p(f).
\end{split}
\end{align}
Combining Equation~(\ref{Eq:Red2}) with Equation~(\ref{Eq:Red-FK})
and applying the entropy/influence conjecture to $g$, we get:
\[
\Ent_p(f) \leq \Ent_{1/2}(g) \leq c \cdot \Inf_{1/2}(g) \leq c
\cdot 6 p \lfloor \log(1/p) \rfloor \Inf_p(f),
\]
and therefore,
\[
\Ent_p(f) \leq c' p \log(1/p) \Inf_p(f),
\]
as asserted.
\end{proof}

\medskip

Consider the random graph model $G(n,p)$. Recall that in this
model, the probability space is $\{0,1\}^N_{p}$, where
$N={{n}\choose{2}}$, the coordinates correspond to the edges of a
graph on $n$ vertices, and each edge exists in the graph with
probability $p$, independently of the other edges. It is
well-known that for the graph property of containing the complete
graph $K_r$ as an induced subgraph, there exists a threshold at
$p_t=\Theta(n^{-2/(r-1)})$. This means that if $p << n^{-2/(r-1)}$
then $\Pr[K_r \subset G | G \in G(n,p)]$ is close to zero, and if
$p >> n^{-2/(r-1)}$ then $\Pr[K_r \subset G | G \in G(n,p)]$ is
close to one. We choose a value $p_0$ in the critical range,
consider the characteristic function $f$ of this graph property in
$G(n,p_0)$, and show that the assertion of
Proposition~\ref{Prop:Biased} is tight for $f$. In order to
simplify the computation, we choose $p_0$ such that the expected
number of copies of $K_r$ in $G(n,p_0)$ is ``nice''. However, the
same argument holds for any value of $p$ in the critical range.
\begin{proposition}
Let $n,r$ be integers such that $r < \log n$. Consider the random
graph $G(n,p_0)$ where $p_0$ is chosen such that ${{n}\choose{r}}
\cdot p_0^{{r}\choose{2}} =1/2$. Let $f$ be defined by:
\[
f(G)=1 \Leftrightarrow \mbox{ G contains a copy of } K_r \mbox{ as
an induced subgraph},
\]
and $f(G)=0$ otherwise. Then
\[
\Ent_{p_0}(f) \geq c \cdot p_0 \log(1/p_0) \cdot \Inf_{p_0}(f),
\]
where $c$ is a universal constant.
\end{proposition}

\begin{proof}
The result is a combination of an upper bound on $\Inf_{p_0}(f)$
with a lower bound on $\Ent_{p_0}(f)$.

In order to bound $\Inf_{p_0}(f)$ from above, note that a
necessary (but not sufficient) condition for an edge $e=(v,w)$ to
be pivotal for $f$ at a graph $G$ \footnote{An edge $e$ is pivotal
for the property $f$ at a graph $G$ if $f(G)=1$ and $f(G \setminus
\{e\})=0$.} is that there exists a set $S$ of $r$ vertices
including $v$ and $w$ such that all ${{r}\choose{2}}$ edges inside
$S$ except for $e$ appear in $G$. Hence, a simple union bound
yields that for any edge $e$,
\[
 \Inf^{p_0}_e(f) \leq {{n-2}\choose{r-2}} \cdot p_0^{{{r}\choose{2}}-1} = \frac{r(r-1)}{n(n-1)p_0}
 \cdot {{n}\choose{r}} p_0^{{{r}\choose{2}}} = \frac{r(r-1)}{2n(n-1)p_0},
 \]
and thus,
\begin{equation}\label{Eq:K_r-1}
\Inf_{p_0}(f) = \sum_{e} \Inf^{p_0}_e(f) \leq \frac{1}{p_0} \cdot
\frac{r(r-1)}{4}.
\end{equation}

In order to bound $\Ent_{p_0}(f)$ from below, we show that at
least a constant portion of the Fourier weight of $f$ is
concentrated on coefficients that correspond to copies of $K_r$ in
$\{0,1\}^N$. Concretely, we show that if $S$ corresponds to a copy
of $K_r$, then:
\begin{equation}\label{Eq:K_r-2}
\hat f(S)^2 \geq c' \cdot {{n}\choose{r}}^{-1}.
\end{equation}
As the number of such coefficients is ${{n}\choose{r}}$, it will
follow that:
\begin{equation}\label{Eq:K_r-3}
\Ent_{p_0}(f) \geq \sum_{\{S: \mbox{ S is a copy of } K_r\}} \hat
f(S)^2 \log \left( \frac{1}{\hat f(S)^2} \right) \geq c' \cdot
\log {{n}\choose{r}} \geq c'' \cdot r \log (n),
\end{equation}
where the rightmost inequality holds since $r<\log n$. Finally, a
combination of Equation~(\ref{Eq:K_r-3}) with
Equation~(\ref{Eq:K_r-1}) will imply:
\[
\Ent_{p_0}(f) \geq c'' \cdot r \log (n) \geq c'' \cdot
\frac{r(r-1)}{2} \cdot \log(1/p_0) \geq c'' \cdot p_0 \log(1/p_0)
\cdot \Inf_{p_0}(f),
\]
as asserted.

To prove Equation~(\ref{Eq:K_r-2}), consider a specific copy $H$
of $K_r$ and denote its set of edges by $S=E(H)$. By the
definition of the Fourier coefficients, we have:
\begin{align}\label{Eq:K_r-4}
\begin{split}
\hat f(S) = &\sum_{T \in \{0,1\}^N} \mu_{p_0}(T) \left(
-\sqrt{\frac{1-p_0}{p_0}} \right)^{|S \cap T|}
\left( \sqrt{\frac{p_0}{1-p_0}} \right)^{{{r}\choose{2}}-|S \cap T|} f(T) \\
= &\sum_{T \in \{0,1\}^N} \mu_{p_0}(T \setminus S) \cdot
(p_0(1-p_0))^{{{r}\choose{2}}/2} (-1)^{|S \cap T|} f(T) \\
= &(p_0(1-p_0))^{{{r}\choose{2}}/2} \sum_{T \in \{0,1\}^N}
\mu_{p_0}(T \setminus S) (-1)^{|S \cap T|} f(T),
\end{split}
\end{align}
where $\mu_{p_0}(T \setminus S)$ denotes the induced measure of
the graph $T \setminus S$. Note that the total contribution to
$\hat f(S)$ of $\{T \in \{0,1\}^N: S \subset T\}$ is
\begin{equation}\label{Eq:K_r-6}
(-1)^{|S|} \cdot (p_0(1-p_0))^{{{r}\choose{2}}/2},
\end{equation}
since $f(T)=1$ for all $T \supset S$. On the other hand, if
$f(T)=1$ and $T \supsetneq S$, then $T$ contains a copy of $K_r$,
in which $k \leq r-1$ vertices are included in $V(H)$, and the
remaining $r-k$ vertices are not included in $V(H)$. Hence, the
total contribution to $\hat f(S)$ of $\{T \in \{0,1\}^N: S
\subsetneq T\}$ is bounded from above (in absolute value) by:
\begin{equation}\label{Eq:K_r-7}
(p_0(1-p_0))^{{{r}\choose{2}}/2} \cdot \sum_{k=0}^{r-1}
{{n}\choose{r-k}} {{r}\choose{k}}
p_0^{{{r}\choose{2}}-{{k}\choose{2}}} =
(p_0(1-p_0))^{{{r}\choose{2}}/2} \cdot (1/2+o_n(1)),
\end{equation}
since for our choice of $p_0$, the term corresponding to $k=0$
equals ${{n}\choose{r}} p_0^{{{r}\choose{2}}} = 1/2$, and the
other terms are negligible. Combining estimates~(\ref{Eq:K_r-6})
and~(\ref{Eq:K_r-7}), we get:
\begin{equation}
\hat f(S)^2 \geq (1-1/2-o_n(1))^2 (p_0(1-p_0))^{{r}\choose{2}}
\geq c p_0^{{r}\choose{2}} = (c/2) \cdot {{n}\choose{r}}^{-1}.
\end{equation}
This completes the proof.
\end{proof}

\medskip

We conclude this section by noting that if $p$ is inverse polynomially small as a
function of $n$, then one can easily prove a statement which is only
slightly weaker than the entropy/influence conjecture. In~\cite{Ryan}
it was shown that with respect to the uniform measure, we have
\[
\Ent_{1/2}(f) \leq (\log n +1) \Inf_{1/2}(f)+1,
\]
for any Boolean $f$.
The statement generalizes easily to a general
biased measure $\mu_p$, and yields the following:
\begin{claim}\label{Claim:Biased}
There exists a universal constant $c$ such that for any $0<p<1$,
for any $n$ and for any $f:\{0,1\}^n_p \rightarrow \{-1,1\}$, we
have
\[
\Ent_p(f) \leq c p(1-p) \log(n) \cdot \Inf_p(f).
\]
\end{claim}

\begin{proof}
Assume w.l.o.g. that $p \leq 1/2$. As shown in~\cite{Ryan}, we have:
\[
\Ent_p(f) \leq (\log n +1) \sum_S |S| \hat f(S)^2 + \epsilon
\log(1/\epsilon)+2 \epsilon,
\]
where $1-\epsilon=\hat f(\emptyset)^2$. (Note that this part of
the proof of Theorem~3.2 in~\cite{Ryan} holds without any change
for the biased measure). In order to bound the term $\epsilon
\log(1/\epsilon)+2 \epsilon$, it was shown in Proposition~3.6
of~\cite{Ryan} that by the edge isoperimetric inequality on the
cube, it is bounded from above by $2\Inf_{1/2}(f)$. By
Equation~(\ref{Eq:Red-FK}), this implies that for the measure
$\mu_p$, we have
\[
\epsilon \log(1/\epsilon)+2 \epsilon \leq 12 p \lfloor \log(1/p)
\rfloor \Inf_p(f)
\]
(since the reduction from the biased measure to the uniform measure preserves the
expectation). Thus, by Equation~(\ref{Eq:Intro-new1}),
\[
\Ent_p(f) \leq (\log n +1) \cdot 4p(1-p) \Inf_p(f)+ 12 p \lfloor
\log(1/p) \rfloor \Inf_p(f) \leq cp(1-p) \log(n) \Inf_p(f),
\]
as asserted.
\end{proof}

For $p$ that is inverse polynomially small in $n$, the statement of Claim~\ref{Claim:Biased}
differs from the assertion of the entropy/influence conjecture only by a constant factor.

\section{Functions with a Low Fourier Weight on the High Levels}
\label{sec:high-levels}

In this section we consider the uniform measure $\mu_{1/2}$ on the
discrete cube, and study Boolean functions with a low Fourier
weight on the high levels. In order to simplify the expression of
the Fourier expansion, we replace the domain by $\{-1,1\}^n$. As a
result, the characters are given by the formula
\[
u_{\{i_1,\ldots,i_r\}}(x)=x_{i_1} \cdot x_{i_2} \cdot \ldots \cdot x_{i_r},
\]
and thus, the Fourier expansion of a function is simply its representation
as a multivariate polynomial.

\begin{proposition}
Let $f:\{-1,1\}^n_{1/2} \rightarrow \mathbb{Z}$, such that all the
Fourier weight of $f$ is concentrated on the first $k$ levels.
Then all the Fourier coefficients of $f$ are of the form
\[
\hat f(S) = a(S) \cdot 2^{-k},
\]
where $a(S) \in \mathbb{Z}$. In particular, $\Ent_{1/2}(f) \leq
2k$.
\end{proposition}

\begin{proof}
The proof is by induction on $k$. The case $k=0$ is trivial.
Assume that the assertion holds for all $k \leq d-1$, and let $f$
be a function of Fourier degree $d$ (i.e., all its Fourier
coefficients are concentrated on the $d$ lowest levels). For $1
\leq i \leq n$, let $f^i$ be the discrete derivative of $f$ with
respect to the $i$th coordinate, i.e.,
\[
f^{i}(x_1,\ldots,x_{i-1},x_{i+1},\ldots,x_n)=
\frac{f(x_1,\ldots,x_{i-1},1,x_{i+1},\ldots,x_n)-
f(x_1,\ldots,x_{i-1},-1,x_{i+1},\ldots,x_n)}{2}.
\]
It is easy to see that if $f=\sum_S \hat f(S) u_S$, then
the Fourier expansion of $f^{i}$ is given by:
\begin{equation}\label{Eq3.1}
f^{i} = \sum_{S \subset (\{1,2,\ldots,n\} \setminus \{i\})}
\hat f(S \cup \{i\}) u_S.
\end{equation}
Hence, $f^{i}$ is of Fourier degree at most $d-1$. Note that by
the definition of $f^{i}$, we have $2f^i(x) \in \mathbb{Z}$ for
all $x \in \{-1,1\}^{n-1}$, and thus by the induction hypothesis,
the Fourier coefficients of $f^{i}$ satisfy $2 \hat f^{i}(S) = a(S)
\cdot 2^{-d+1}$, where $a(S) \in \mathbb{Z}$. This holds for any
$1 \leq i \leq n$, and therefore, by Equation~(\ref{Eq3.1}), all
the Fourier coefficients of $f$ (except, possibly, for $\hat
f(\emptyset)$), are of the form $\hat f(S) = a(S) \cdot 2^{-d}$,
where $a(S) \in \mathbb{Z}$. Finally, $\hat f(\emptyset)$ must be
also of this form, since otherwise $f(x)$ cannot be an integer.
This completes the proof.
\end{proof}

In an unpublished work~\cite{Bourgain-Kalai2}, Bourgain and Kalai
obtained a stronger result:
\begin{theorem}\label{Thm:BK}
Let $f:\{-1,1\}^n \rightarrow \{-1,1\}$, and assume that there
exist $c_0>0$, $0<a<1/2$, and $k$, such that for all $t$,
\[
\sum_{\{S:|S|>t\}} \hat f(S)^2 \leq e^{c_0 k} \cdot e^{-at},
\]
then for any $\alpha>1$, there exists a set $B_{\alpha}$, such
that:
\begin{enumerate}
\item $\log |B_{\alpha}| \leq C \cdot \alpha k$, where $C$ depends
only on $a$ and $c_0$.

\item $\sum_{S \not \in B_{\alpha}} \hat f(S)^2 \leq n^{-\alpha}$.
\end{enumerate}
\end{theorem}

The theorem asserts that if the Fourier weight of $f$ beyond the
$k$th level decays exponentially, then most of the Fourier weight
of $f$ is concentrated on $exp(Ck)$ coefficients, and thus,
$\Ent_{1/2}(f) \leq C' k$ (for an appropriate choice of $C'$). The
proof uses the $d$th discrete derivative of $f$ (like our proof
above), and the Bonami-Beckner hypercontractive
inequality~\cite{Bonami,Beckner}. We note that the exact
dependence of $C$ on $a$ (i.e., the rate of the exponential decay)
in the assertion of the theorem, which is important if $a$ is
allowed to be a function of $n$, is of order $C=\Theta(a^{-1} \log
(a^{-1}))$.

\bigskip

\noindent \textbf{A tensorisation technique.}
In~\cite{Kalai-Blog}, Kalai observed that the entropy/influence
conjecture tensorises, in the following sense. For
$f:\{-1,1\}_{1/2}^l \rightarrow \{-1,1\}$ and $g:\{-1,1\}_{1/2}^m
\rightarrow \{-1,1\}$, define $f \otimes g:\{-1,1\}^{l+m}_{1/2}
\rightarrow \{-1,1\}$ by:
\[
f \otimes g(x_1,\ldots,x_{l+m}) = f(x_1,\ldots,x_l) \cdot
g(x_{l+1},\ldots,x_{l+m}).
\]
Furthermore, let
\[
f^{\otimes N} = f \otimes f \otimes \ldots \otimes f,
\]
where the tensorisation is performed $N$ times. It is easy to see
that $\Inf_{1/2}(f^{\otimes N})=N \cdot \Inf_{1/2}(f)$ and
$\Ent_{1/2}(f^{\otimes N}) = N \cdot \Ent_{1/2}(f)$. Hence,
proving the entropy/influence conjecture for any ``tensor power''
of $f$ is equivalent to proving the conjecture for $f$ itself.
This observation was used in~\cite{Ryan} to deduce that it is
sufficient to prove a seemingly weaker version of the conjecture:
$\Ent_{1/2}(f) \leq c \Inf_{1/2}(f) + o(n)$, where $n$ is the
number of variables.

\medskip

We observe that tensorisation can be used to enhance the rate of
decay of the Fourier coefficients. By the Law of Large Numbers, as
$N \rightarrow \infty$, the level of the Fourier coefficients of
$f^{\otimes N}$ is concentrated around its expectation, which is
$N \cdot \Inf_{1/2}(f)$, and the rate of decay above that level,
i.e., $\sum_{|S|>t} \widehat{f^{\otimes n}}(S)^2$ becomes
``almost'' inverse exponential in $t$. This holds even if the rate
of decay of the Fourier coefficients of $f$ is much slower (like
the Majority function, for which $\sum_{|S|>t} \widehat{MAJ}(S)^2
\approx t^{-1/2}$). Therefore, if one obtains a result similar to
Bourgain-Kalai's Theorem~\ref{Thm:BK} for a slower rate of decay,
e.g., under the weaker assumption $\sum_{|S|>t} \hat{f}(S)^2 \leq
e^{c_0 \sqrt{k}} \cdot e^{-a \sqrt{t}}$, then the result can be
enhanced to any rate of decay, by tensoring the function to itself
until its rate of decay reaches $e^{-a \sqrt{t}}$. However, we
weren't able to find such generalization of the Bourgain-Kalai
result.

\section{Concluding Remarks}
\label{sec:remarks}

We conclude this paper with two remarks related to the
entropy/influence conjecture.

\medskip

\noindent \textbf{A weaker upper bound on the entropy that can be proved easily.}
As mentioned in Section~\ref{sec:biased}, it was shown in~\cite{Ryan}
that with respect to the uniform measure, one can easily prove the
following weaker upper bound on the entropy of any Boolean function:
\[
\Ent_{1/2}(f) \leq (\log n +1) \Inf_{1/2}(f)+1.
\]
We provide an independent proof of a slightly stronger claim.

\begin{claim}\label{Claim:Information}
For any $n$ and for any
$f:\{0,1\}^n_{1/2} \rightarrow \mathbb{R}$, we have
\[
\Ent_{1/2}(f) \leq \sum_{i=1}^n h(\Inf^{1/2}_i(f))  \leq 2
\Inf_{1/2}(f)(\log n - \log \Inf_{1/2}(f)),
\]
where
\[
h(x) = -x \log x - (1-x) \log(1-x).
\]
\end{claim}

\begin{proof}
As the proof deals only with the uniform measure on the discrete
cube, we write $\Ent(f)$ and $\Inf(f)$ instead of $\Ent_{1/2}(f)$
and $\Inf_{1/2}(f)$ during the proof.

Let $S \subset \{1,2,\ldots,n\}$ be chosen according to the
Fourier distribution (i.e., $\Pr[S=S_0]=\hat f(S_0)^2$), and let
$X_i = 1_{\{i \in S\}}$. Then by the basic rules of entropy,
\[
\Ent(f) = H(S) = H(X_1,\ldots,X_n) \leq \sum_{i=1}^n H(X_i) =
\sum_{i=1}^n h(\Inf_i(f)),
\]
thus obtaining the first inequality. Note that if $\Inf_i(f) \geq
0.5$, then $h(\Inf_i(f)) \leq 2 \Inf_i(f)$, and otherwise,
$h(\Inf_i(f)) \leq -2 \Inf_i(f) \log \Inf_i(f)$. Therefore,
\begin{align*}
\frac{1}{2} \Ent(f) &\leq \Inf(f) + \sum_{i=1}^n
\Inf_i(f) (-\log \Inf_i(f)) \\
&= \Inf(f) \left(1 + \sum_{i=1}^n \frac{\Inf_i(f)}{\Inf(f)}
 \cdot \big(-\log \frac{\Inf_i(f)}{\Inf(f)} - \log \Inf(f)\big) \right).
\end{align*}
 We note that the expression $\sum_{i=1}^n \Inf_i(f)/\Inf(f) (-\log \Inf_i(f)/\Inf(f))$
 is the entropy of the random variable $Y$ defined by $\Pr[Y=i]=\Inf_i(f)/\Inf(f)$ which is
 supported on $\{1,2,\ldots,n\}$, and is therefore bounded by $\log n$. We thus conclude that
\[
\frac{1}{2} \Ent(f) \leq \Inf(f) (1 + \log n - \log \Inf(f))
\]
as asserted.
\end{proof}

It is easy to see that the bound using the entropy is stronger in some cases, in
particular when there is variability in the influences of different coordinates.

We note that the proof does not use the fact that $f$ is Boolean
and indeed it could not provide a proof of the Entropy/Influence
conjecture, as can be seen, e.g., for the majority function, where
$\Inf_{1/2}(f)$ is of order $\sqrt{n}$ while $\sum_{i=1}^n
h(\Inf^{1/2}_i(f))$ is of order $\sqrt{n} \log n$.

\medskip

\noindent \textbf{Relation to Friedgut's characterization of
functions with a low influence sum.} In~\cite{Friedgut1}, Friedgut
showed that any Boolean function $f:\{0,1\}^n_{p} \rightarrow
\{0,1\}$ essentially depends on at most $C(p)^{\Inf(f)}$
coordinates, where $C(p)$ depends only on $p$. The main step of
the proof is to show that most of the Fourier weight of the
function is concentrated on sets that contain one of these
coordinates. A stronger claim one may hope to prove is that most
of the Fourier weight is concentrated on at most $C(p)^{\Inf(f)}$
coefficients. Formally, we raise the following conjecture that
resembles the assertion of Bourgain-Kalai's theorem:
\begin{conjecture}
For any $0<p<1$, there exists a constant $C(p)>0$ such that for
any $\epsilon>0$, for any $n$ and for any $f:\{-1,1\}^n_p
\rightarrow \{-1,1\}$, there exists a set $B_\epsilon \subset
\{0,1\}^n$ such that:
\begin{enumerate}
\item $\log |B_{\epsilon}| \leq C(p) \cdot \Inf(f)$, and

\item $\sum_{S \not \in B_{\epsilon}} \hat f(S)^2 < \epsilon$.
\end{enumerate}
\end{conjecture}

This conjecture is clearly stronger than Friedgut's theorem and
even implies a variant of Mansour's conjecture~\cite{Mansour}
(since as shown in~\cite{Boppana}, if a Boolean function $f$ can
be represented by an $m$-term DNF, then $\Inf_{1/2}(f) = O(\log
m)$ ), but it still does not imply the entropy/influence
conjecture, since the remaining Fourier coefficients (whose total
Fourier weight is at most $\epsilon$) can still contribute $n
\cdot \epsilon$ to $\Ent_{1/2}(f)$.

\section{Acknowledgements}

We are grateful to Gil Kalai for introducing us to his unpublished
work with Jean Bourgain~\cite{Bourgain-Kalai2}, and to Ryan
O'Donnell for sending us his recent work~\cite{Ryan}.

\end{document}